\documentclass[12pt,reqno]{article}

\usepackage[usenames]{color}
\usepackage{amssymb}
\usepackage{graphicx}
\usepackage{amscd}

\usepackage[colorlinks=true,
linkcolor=webgreen,
filecolor=webbrown,
citecolor=webgreen]{hyperref}

\definecolor{webgreen}{rgb}{0,.5,0}
\definecolor{webbrown}{rgb}{.6,0,0}

\usepackage{color}
\usepackage{fullpage}
\usepackage{float}

\usepackage{graphics,amsmath,amssymb}
\usepackage{amsthm}
\usepackage{amsfonts}
\usepackage{latexsym}
\usepackage{epsf}

\def\Enn{\mathbb{N}}
\def\Zee{\mathbb{Z}}

\DeclareMathOperator\cl{CL}

\title{Length of the continued logarithm algorithm on rational inputs}

\author{Jeffrey Shallit \\
School of Computer Science \\
University of Waterloo \\
Waterloo, ON  N2L 3G1 \\
Canada \\
{\tt shallit@cs.uwaterloo.ca} }

\begin{document}

\maketitle

\theoremstyle{plain}
\newtheorem{theorem}{Theorem}
\newtheorem{corollary}[theorem]{Corollary}
\newtheorem{lemma}[theorem]{Lemma}
\newtheorem{proposition}[theorem]{Proposition}

\theoremstyle{definition}
\newtheorem{definition}[theorem]{Definition}
\newtheorem{example}[theorem]{Example}
\newtheorem{conjecture}[theorem]{Conjecture}

\theoremstyle{remark}
\newtheorem{remark}[theorem]{Remark}

\begin{abstract}
The continued logarithm algorithm was introduced by Gosper around
1978, and recently studied by Borwein, Calkin, Lindstrom, and
Mattingly.  In this note I show that the continued logarithm
algorithm terminates in at most $2 \log_2 p + O(1)$ steps on input
a rational number $p/q \geq 1$.  Furthermore, this bound is tight,
up to an additive constant.
\end{abstract}

\section{Introduction}

Let $\Zee$ denote the integers,
$\Enn$ denote the non-negative integers, and
$\Enn_{>0}$ denote the positive integers.

The continued fraction algorithm, which expands every real
number $x$ in an expression of the form
\begin{equation*}
x = a_0+{1\over\displaystyle a_1+
	{\strut 1\over\displaystyle a_2+
	  {\strut 1\over\displaystyle a_3+ 
\raisebox{-1ex}{$\ddots$} }}},
\end{equation*}
with $a_0 \in \Zee$ and $a_i \in \Enn_{>0}$ for $i \geq 1$,
has been extensively studied, in part because of its relationship
to the Euclidean algorithm.  In particular, it is known that this
expression is essentially unique,
and terminates with final term $a_n$ if and only if $x$ is a rational
number.  In this case, if $x = p/q$, 
the length of the expansion
is at most $O(\log pq)$, as has been known since the 1841
work of Finck \cite{Shallit:1994}.
Furthermore, examples
achieving this bound are known.

Around 1978, Gosper \cite{Gosper:1978} introduced an analogue of the
continued fraction algorithm for real numbers, called the {\it continued
logarithm algorithm}, which expands
every real number $x \geq 1$ in an expression of the form
\begin{equation}
x = 2^{k_0}(1 + {1\over\displaystyle 2^{k_1} (1 +
        {\strut 1\over\displaystyle 2^{k_2} ( 1 +
	          {\strut 1\over\displaystyle 2^{k_3} (1 + 
		  \raisebox{-1ex}{$\ddots$}) } ) } ) } ) \quad .
\label{clog}
\end{equation}
where $k_i \in \Enn$ for $i \geq 0$.
More recently, this algorithm was studied by
Brabec \cite{Brabec:2006,Brabec:2007,Brabec:2010} and
Borwein, Calkin, Lindstrom, and Mattingly \cite{Borwein:2016}.

Once again, it is known that this expression is essentially unique
and
terminates if and only if
$x$ is rational.  However, up to now, no estimate of the length of the
expansion has been given.  In this note, we provide such an estimate.

\section{The continued logarithm algorithm}

As described by Borwein,  Calkin, Lindstrom, and Mattingly \cite{Borwein:2016},
the continued logarithm algorithm can be described as follows:
for $x > 1$ we define
$$ g(x) = \begin{cases} 
	{x \over 2}, & \text{if $x \geq 2$}; \\
	{1 \over {x-1}}, & \text{if $1 < x < 2$}.
	\end{cases}
$$
The algorithm proceeds by iterating $g$ until the result is $1$;
the division steps $x \rightarrow x/2$ are done repeatedly until the 
transformation $x \rightarrow {1 \over {x-1}}$ is used, or $x = 1$.  In the latter
case the algorithm terminates.  The number of division steps is given
by the number of $k$'s in the expansion (\ref{clog}).
For example, 
\begin{equation}
{{96} \over 7}  = 2^{3}(1 + {1\over\displaystyle 2^{0} (1 +
        {\strut 1\over\displaystyle 2^{1} ( 1 +
	          {\strut 1\over\displaystyle 2^{2} })})}) \quad .
\end{equation}
We can abbreviate the expression (\ref{clog}) by
writing $x = \langle k_0, k_1, \ldots, k_n \rangle $.
So ${{96} \over 7} = \langle 3,0,1,2 \rangle$.  
Similarly, $2^k = \langle k \rangle$ for $k \geq 0$.  Since the
continued logarithm expansion
is unique, we can regard an expression like
$x = \langle k_0, k_1, \ldots, k_n \rangle $ as {\it either\/} an evaluation
of a certain function on the right, or as a statement about
the output of the continued
logarithm algorithm on an input $x$.  We trust there will be no confusion
on the proper interpretation in what follows.

There are two different natural measures of the complexity of the
algorithm on rational inputs.  The first is the number of steps $n+1$ in
$x = \langle k_0, k_1, \ldots, k_n \rangle $, which we
write as $L(x)$.    The second is the {\it total\/} number of division
steps $k_0 + k_1 + \cdots + k_n$, which we write as $T(x)$.
In this note we get asymptotically tight bounds for $L$ and $T$
on rational numbers $p/q \geq 1$.

\section{The bound on $L$}

Consider performing the continued logarithm algorithm $\cl$ on a
rational input ${p \over q} \geq 1$, getting back
${p \over q} = \langle k_0, k_1, \ldots, k_n \rangle $.
We can associate a rational number ${p \over q}$ with the pair $(p,q)$.
While this association is not unique (for example, $2$ can be represented
by $(2,1)$ or $(4,2)$), it does not create problems in what follows.
By consolidating the division steps, we can
express the continued logarithm algorithm on rational numbers as
a function of two integers that takes its value on finite lists,
as follows
\begin{equation}
\cl (p,q)  = \begin{cases}
	k, & \text{if $p = 2^k q$ for some $k \geq 0$}; \\
	k, \cl(2^k q, p - 2^k q),
		& \text{if $1 < {p \over {2^k q}} < 2$}; 
	\end{cases}
\end{equation}
Here the comma denotes concatenation.

The idea of our bound on $L$ is to consider how the measure $f(p,q) =
p^2 + q^2$ changes as the algorithm proceeds.

In our interpretation of the algorithm on pairs $(p,q)$,
it replaces $(p,q)$ with
$(p', q')$, where $p' = 2^k q$ and $q' = p - 2^k q$ for
$1 \leq {p \over {2^k q}} < 2$, and terminates when $q = 0$.
First we show that $f(p,q)$ strictly decreases in
each step of the algorithm. 

\begin{lemma}
If the continued logarithm algorithm takes $(p,q)$ to $(p',q')$, then
$f(p', q') < f(p,q)$.
\end{lemma}

\begin{proof}
From the inequality ${p \over {q \cdot 2^k}} \geq 1$ we get
$$ {p \over q} \geq 2^k > 2^k - {1 \over {2^{k+1}}} .$$
Multiplying by
$2^{k+1} q^2$, we get
$ 2^{k+1} pq > (2^{2k+1} - 1) q^2 $.
Adding $p^2$ to both sides and rearranging gives
\begin{align*}
p^2 + q^2 &> p^2 - 2^{k+1} pq + 2^{2k+1} q^2  \\
&= (2^k q)^2 + (p-2^k q)^2 \\
& = (p')^2 + (q')^2,
\end{align*}
as desired.
\label{one}
\end{proof}

Next we show how $f$ decreases as the algorithm proceeds.
We use the notation $(p,q) \rightarrow^k (p', q')$ to denote that
one step of the algorithm replaces $p/q$ with $p'/q'$, where
$p' = 2^k q$ and $q' = p - 2^k q$.

\begin{lemma}
If $(p,q) \rightarrow^0 (p', q')$
then $f(p',q') \leq f(p, q)/2$.
\label{step0}
\end{lemma}

\begin{proof}
The condition $k = 0$ implies
$p' = q$ and $q' = p - q$.
Then $ 1 \leq {p \over q} < 2$.  If ${p \over q} = 1$ then
the algorithm terminates, so assume ${ p \over q } > 1$.
Write ${p \over q} = c$.
If $ 1 < c \leq 2$, then $(c-1)(c-3) < 0$.  
Multiplying by $q^2$ gives
$q^2(c^2 - 4c + 3) < 0$.
Hence, using the fact that $p = cq$, we get
$p^2 - 4pq + 3q^2 < 0$.    Dividing by $2$ and rearranging,
we get
$f(p', q') = q^2 + (p-q)^2 < {1 \over 2} (p^2 + q^2) < {1 \over 2} f(p,q)$.
\end{proof}

\begin{lemma}
If two steps of the continued logarithm algorithm are
$(p,q) \rightarrow^k (p', q') \rightarrow^0 (p'', q'')$ with $k \geq 1$
then $f(p'', q'') < {1 \over 4} f(p,q)$.
\label{stepk0}
\end{lemma}

\begin{proof}
The first step implies that $1 < {p \over {2^k q}} < 2$, and
$p' = 2^k q$, $q' = p - 2^k q$.  The second step implies that
$1 < {{2^k q} \over {p - 2^k q}} < 2$ and
$p'' = p - 2^k q$, $q'' = 2^{k+1} q - p$.

Define $c = {p \over {2^k q}}$ and observe that the inequalities
of the previous paragraph imply $3/2 \leq c \leq 2$.
Consider the polynomial $h(c) = 7c^2 -24c + 20$.  Since the roots
of this polynomial are $10/7$ and $2$, we clearly 
see that $h(c) \leq 0$ for $3/2 \leq c \leq 2$.  Hence
$q^2 (7c^2 - 24c + 20) 2^{2k} < q^2$ for $q \geq 1$ and $k \geq 1$.
Substituting $c = {p \over {2^k q}}$ and simplifying gives
$7p^2 - 24 \cdot 2^k pq + 20 \cdot 2^{2k} q^2 < q^2$.  Adding $p^2$ to both
sides and then dividing by $4$ gives
$2p^2 - 6 \cdot 2^k pq + 5 \cdot 2^{2k} q^2 < {1 \over 4}(p^2 + q^2)$.
But the left side of this inequality is
$(p-2^k q)^2 + (q 2^{k+1} - p)^2$.   Thus we have proved
$(p'')^2 + (q'')^2 < {1 \over 4}(p^2 + q^2)$.
\end{proof}

\begin{lemma}
If two steps of the continued logarithm algorithm are
$(p,q) \rightarrow^k (p', q') \rightarrow^{k'} (p'', q'')$
with $k, k' \geq 1$, then $p''$ and $q''$ are both divisible by $2$, and
$f(p''/2, q''/2) < {1 \over 4} f(p,q)$.
\label{stepkk}
\end{lemma}

\begin{proof}
In that case the
first step replaces $(p,q)$ by $(p', q') = (q\cdot 2^k, p- q \cdot 2^k)$, and
the second step replaces this latter pair with
$$(p'', q'') =
( (p- q \cdot 2^k) 2^{k'}, q (2^{k+k'} + 2^k) - p \cdot 2^{k'} ).$$
Now both elements of this latter pair are divisible by $2^{\min(k,k')} \geq 2$.
Since these correspond to numerator and denominator, we can
divide both elements of the pair by $2$ and obtain an equivalent pair of
integers $(p''/2, q''/2)$.  Note that 
$\cl(p'', q'') = \cl(p''/2, q''/2)$.
By Lemma~\ref{one} we have $f(p''/2, q''/2) = {1 \over 4} f(p'', q'') <
{1 \over 4} f(p',q') < {1 \over 4} f(p,q)$, as desired.
\end{proof}

\begin{theorem}
On input $p/q \geq 1$ the continued logarithm algorithm uses
at most $2 \log_2 p + 2$ steps.
\label{bnd}
\end{theorem}

\begin{proof}
Consider the continued logarithm expansion and process it from left
to right as follows:  if a term is $0$, use Lemma~\ref{step0}.
If a term is $ k \geq 1$, group it with the term that follows and use
either Lemma~\ref{stepk0} or Lemma~\ref{stepkk}.  By doing so we
group all terms except possibly the last.  Lemma~\ref{step0} shows
that a single step reduces $f$ by a factor of $2$.
Lemmas~\ref{stepk0} and \ref{stepkk} show that two steps reduce $f$
by a factor of $4$.  Thus the total number of steps on input
$(p,q)$ is at most $\log_2 (p^2 + q^2) + 1$, where the $+1$ term
takes into account the last term that might be ungrouped.

So the algorithm uses at most $\log_2 (p^2 + q^2) + 1$ steps.
Since $p \geq q$, we have
$\log_2 (p^2 + q^2) + 1 \leq \log_2 (2p^2)  + 1
\leq \log_2 (p^2) + 2 \leq (2 \log_2 p) + 2$.
\end{proof}

A nearly matching lower bound of $2 \log_2 p + O(1)$ is achievable,
as the following
class of examples shows:

\begin{theorem}
On input $2^n - 1$ the continued logarithm algorithm takes $2n-2$ steps.
\end{theorem} 

\begin{proof}
We have $2^n - 1 = \langle n-1, 0, n-2, 0, \ldots, 2, 0, 1, 1 \rangle$,
as can be easily proved by induction.
\label{power2}
\end{proof}

\section{The bound on $T$}

\begin{theorem}
Let ${p \over q} \geq 1$.  Then $T({p \over q}) < (\log_2 p) (2\log_2 p + 2)$.
\end{theorem}

\begin{proof}
As the continued logarithm 
algorithm proceeds, the numerators strictly decrease, so
each $k$ is bounded by $\log_2 p$.  The number of steps is bounded by
Theorem~\ref{bnd}.
\end{proof}

\begin{theorem}
For $n \geq 1$ we have $T(2^n - 1) = n(n-1)/2 +1$.
\end{theorem}

\begin{proof}
Follows immediately from the expansion
$2^n - 1 = \langle n-1, 0, n-2, 0, \ldots, 2, 0, 1, 1 \rangle$
given in the proof of Theorem~\ref{power2}.
\end{proof}

\section{Open problems}

1.  What is the average case behavior of the number of steps of the
continued logarithm algorithm
on rational numbers $p/q$, with $q < p < 2q$, as $q \rightarrow \infty$?

2.  Is the sequence $(L(n))_{n \geq 1}$ a $k$-regular sequence for any
$k \geq 2$?  The available numerical evidence suggests not.


\begin{thebibliography}{9}

\bibitem{Borwein:2016}
Jonathan M. Borwein,
Neil J. Calkin,
Scott B. Lindstrom,
and
Andrew Mattingly.
\newblock Continued logarithms and associated continued fractions.
\newblock Preprint, May 11 2016, available at
\url{https://www.carma.newcastle.edu.au/jon/clogs.pdf}.

\bibitem{Brabec:2006}
Tom{\'a}{\v s} Brabec.
\newblock Hardware implementation of continued logarithm arithmetic.
\newblock In {\it Scientific Computing, Computer Arithmetic and Validated Numerics, 2006. SCAN 2006}, IEEE, 2006, pp.~1--9.

\bibitem{Brabec:2007}
Tom{\'a}{\v s} Brabec.
\newblock On progress of investigations in continued logarithms.
\newblock Preprint.  Available at
\url{http://citeseerx.ist.psu.edu/viewdoc/download?doi=10.1.1.93.4552&rep=rep1&type=pdf}.

\bibitem{Brabec:2010}
Tom{\'a}{\v s} Brabec.
\newblock Speculatively redundant continued logarithm representation.
\newblock {\it IEEE Trans. Computers} {\bf 59} (2010), 1441--1454.

\bibitem{Gosper:1978}
Bill Gosper.
\newblock Continued fraction arithmetic.
\newblock Unpublished manuscript, c. 1978.  Available at
\url{http://perl.plover.com/classes/cftalk/INFO/gosper.txt} or
\url{http://www.tweedledum.com/rwg/cfup.htm}.

\bibitem{Shallit:1994}
J. Shallit.
\newblock Origins of the analysis of the Euclidean algorithm.
\newblock {\it Historia Math.} {\bf 21} (1994), 401--419.

\end{thebibliography}
\end{document}